\documentclass[12pt, reqno]{amsart}
\usepackage{amsmath, amsthm, amscd, amsfonts, amssymb, graphicx, color}
\usepackage[bookmarksnumbered, colorlinks, plainpages]{hyperref}
\hypersetup{colorlinks=true,linkcolor=red, anchorcolor=green, citecolor=cyan, urlcolor=red, filecolor=magenta, pdftoolbar=true}

\usepackage{tensor}

\newtheorem{theorem}{Theorem}[section]
\newtheorem{lemma}[theorem]{Lemma}
\newtheorem{proposition}[theorem]{Proposition}
\newtheorem{corollary}[theorem]{Corollary}
\theoremstyle{definition}
\newtheorem{definition}[theorem]{Definition}
\newtheorem{example}[theorem]{Example}

\theoremstyle{remark}

\numberwithin{equation}{section}

\usepackage{fullpage}

\begin{document}

\setcounter{page}{1}

\title[FPT on $G$-metric spaces]{Some extensions of Banach'\ CP in $G$-metric spaces}

\author[Ya\'e Olatoundji Gaba]{Ya\'e Olatoundji Gaba$^{1,2,*}$}

%\address{$^{1}$Department of Mathematics and Applied Mathematics, University of Cape Town, South Africa.}
%\email{\textcolor[rgb]{0.00,0.00,0.84}{gabayae2@gmail.com
%}}

\address{$^{1}$\'Ecole Normale Sup\'erieure de Natitingou, Universit\'e de Parakou, B\'enin.}

\address{$^{2}$Institut de Math\'ematiques et de Sciences Physiques (IMSP)/UAC,
Porto-Novo, B\'enin.}

\address{$^{*}$\textit{Corresponding author.}}

\email{\textcolor[rgb]{0.00,0.00,0.84}{gabayae2@gmail.com
}}



\begin{abstract}
We present different extensions of the Banach contraction principle in the $G$-metric space setting. More precisely, we consider mappings for which the contractive condition is satisfied by a power of the mapping and for which the power depends on the specified point in the space. We first state the result in the continuous case and later, show that the continuity is indeed not necessary. Imitating some techniques obtained in the metric case, we prove that under certain conditions, it is enough for the contractive condition to be verified on a proper subset of the space under consideration. These results generalize well known comparable results.
\end{abstract} 

\maketitle

\section{Introduction and Preliminaries}
Many generalizations of metric spaces have appeared in the last decades and even most of them are metric-like spaces (see cone metric spaces, $G$-metric spaces, etc.), they present they own paricularity and address some specific mathematical problems, both in theory and in application. Once an expos\'e\footnote{At least on a topological point of view.} of the space has been done, it is mathematically ideal to investigate the behaviour of maps between these spaces, specially self maps that leave certain points of the space fixed. The Banach Contraction Principle (BCP) is surely the most celebrated result in fixed point theory. Therefore it represented the default starting point for fixed point theory in different generalized metric spaces. The BCP
has been generalized in many different directions, in many different generalized metric spaces.
The generalized metric space which is our focus here is the $G$-metric space.  Different extensions of the BCP were also presented in $G$-metric spaces, so inspired by the works of Mustafa \cite{mustafa1}, Bryant \cite{t} and Gaji\'c \cite{Gaj}, we prove the following:

\begin{theorem}\label{BCPext}

Let $(X,G)$ be a complete $G$-metric space and let $T:X\to X$
 be a continuous mapping satisfying the condition: there exists $\lambda\in [0,1)$ such that for each $x\in X$, there exists a positive integer $n(x)$ such that
 \begin{equation}\label{BCPexteq}
 G(T^{n(x)}x,T^{n(x)}y,T^{n(x)}z) \leq \lambda \ G(x,y,z),
 \end{equation}
 whenever $y,z\in X.$ 
 Then $T$ leaves exactly one point of $X$ fixed.
\end{theorem}

The work we present via the different theorems we proved, also extend some theorems of well-known authors such as of \'{C}iri\'{c} \cite{ciric}, Jachymaski \cite{ja1}, Rhoades \cite{ro}, from metric spaces to $G$-metric spaces. Similar work
can also be read in \cite{Gaj} and the references therein.
Also a few recent results about fixed point in $G$-metric spaces can be read in \cite{Gaba1,Gaba5}.
%j,v,r
The basic ideas about $G$-metrics can be read in \cite{Mustafa} but for the convenience of the reader, we here recall the most important ones.

\begin{definition} (Compare \cite[Definition 3]{Mustafa})
Let $X$ be a nonempty set, and let the function $G:X\times X\times X \to [0,\infty)$ satisfy the following properties:
\begin{itemize}
\item[(G1)] $G(x,y,z)=0$ if $x=y=z$ whenever $x,y,z\in X$;
\item[(G2)] $G(x,x,y)>0$ whenever $x,y\in X$ with $x\neq y$;
\item[(G3)] $G(x,x,y)\leq G(x,y,z) $ whenever $x,y,z\in X$ with $z\neq y$;
\item[(G4)] $G(x,y,z)= G(x,z,y)=G(y,z,x)=\ldots$, (symmetry in all three variables);

\item[(G5)]
$$G(x,y,z) \leq [G(x,a,a)+G(a,y,z)]$$ for any points $x,y,z,a\in X$.
\end{itemize}
Then $(X,G)$ is called a \textbf{$G$-metric space}.

\end{definition}

\begin{proposition}\label{prop1} (Compare \cite[Proposition 6]{Mustafa})
Let $(X,G)$ be a $G$-metric space. 
%Define on $X$ the metric  $d_G$ by $d_G(x,y)= G(x,y,y)+G(x,x,y)$ whenever $x,y \in X$. 
Then for a sequence $\{x_n\} \subseteq X$, the following are equivalent
\begin{itemize}
\item[(i)] $\{x_n\}$ is $G$-convergent to $x\in X.$

\item[(ii)] $\lim_{n,m \to \infty}G(x,x_n,x_m)=0.$

%\item[(iii)]  $\lim_{n \to \infty}d_G(x_n,x)=0$.

\item[(iii)]$\lim_{n \to \infty}G(x,x_n,x_n)=0.$ 

\item[(iv)]$\lim_{n \to \infty}G(x_n,x,x)=0.$ 
\end{itemize}

\end{proposition}

\begin{proposition}(Compare \cite[Proposition 9]{Mustafa})

In a $G$-metric space $(X,G)$, the following are equivalent
\begin{itemize}
\item[(i)] The sequence $\{x_n\} \subseteq X$ is $G$-Cauchy.

\item[(ii)] For each $\varepsilon >0$ there exists $N \in \mathbb{N}$ such that $G(x_n,x_m,x_m)< \varepsilon$ for all $m,n\geq N$.

\end{itemize}

\end{proposition}

\begin{definition} (Compare \cite[Definition 9]{Mustafa})
 A $G$-metric space $(X,G)$ is complete (or more precisely $G$-complete) if every $G$-Cauchy sequence of elements of $(X,G)$ is $G$-convergent in  $(X,G)$. 

\end{definition}

\begin{proposition}(Compare \cite[Proposition 7]{Mustafa})

If $(X,G)$ and $(X',G')$ are two $G$-metric space, then a function $T:(X,G)\to (X',G')$ is continuous at a point $x^*\in X$ if and only if whenever a sequence $\{x_n\}\subseteq X$ is $G$-convergent to $x^*\in X$, then the sequence $\{Tx_n\}\subseteq X'$ is $G'$-convergent to $Tx^*\in X'$.

\end{proposition}

We also recall the results by Mustafa:

\begin{theorem}\label{BCP}(\cite{mustafa1})

Let $(X,G)$ be a $G$-complete $G$-metric space and let $T:X\to X$
 be a mapping such that there exists $\lambda\in [0,1)$ satisfying
 
 \begin{equation}\label{BCPeq}
 G(Tx,Ty,Tz) \leq \lambda \ G(x,y,z),
 \end{equation}

whenever $x,y,z\in X.$ 
 Then $T$ has a unique fixed point. In fact, T is a Picard operator.
\end{theorem}

And the result by Bryant:

\begin{theorem}\label{BCPcor}(Compare\cite{t})

Let $(X,d)$ be a complete metric space and let $T:X\to X$
 be a mapping such that there exists $\lambda\in [0,1)$ satisfying
 
 \begin{equation}\label{BCPeqcor}
 d(T^nx,T^ny) \leq \lambda \ d(x,y),
 \end{equation}
for some $n>1$, whenever $x,y\in X.$ 
  Then $T$ has a unique fixed point.
\end{theorem}

\section{First results}

We start with the following lemma, needed for the next Theorem, and for which a similar version has been given by Gaji\'c et al. \cite{Gaj}.
\begin{lemma}\label{lem1}

Let $T$ be a map satifying the conditions of Theorem \ref{BCPext}, then the extended real valued mapping $r:X\to [0,\infty]$ defined by

\[
r(x) = \sup_n G(x, T^nx,T^nx),
\]
is actually real valued, i.e. $r(x)<\infty$ whenever $x\in X.$ 
\end{lemma}

\begin{proof}
Let $x\in X$ and define

\[
l(x) = \max \{ G(x, T^ix,T^ix): i=1,2,\cdots, n(x) \}.
\]

For a positive interger $n$, by the Archimedean property, there exists an integer $s\geq 0$ such that 
\[
s.n(x) < n \leq (s+1).n(x).
\]
Therefore
\begin{align*}
G(x, T^nx,T^nx) & \leq G(x, T^{n(x)}x,T^{n(x)}x) +G(T^{n(x)}x,T^{n(x)}T^{n-n(x)}x,T^{n(x)}T^{n-n(x)}x)  \\ 
               & \leq  l(x) + \lambda G(x,T^{n-n(x)}x,T^{n-n(x)}x) \\
               & \leq l(x) + \lambda l(x) + \lambda^2 l(x)+\cdots +\lambda^s l(x)\\
               & \leq \frac{l(x)}{1-\lambda} \ \text{ for all } n\geq 0.
\end{align*}
Hence $r(x)<\infty$ whenever $x\in X$.

\end{proof}

\subsection{Proof of the Theorem \ref{BCPext}}

\begin{proof}
Let $x_0\in X$ be arbitrary. We construct the sequence $\{x_n\}$ inductively by setting
\[
x_1 = T^{n(x_0)}(x_0), \text{ and } x_{i+1} = T^{n(x_i)}(x_i).
\]
If wet set $m_i= n(x_i)$, by usual procedure, we have that
\begin{align*}
G(x_n, x_{n+1},x_{n+1}) & = G(T^{m_{n-1}}x_{n-1},T^{m_{n-1}}T^{m_n}x_{n-1} ,T^{m_{n-1}}T^{m_n}x_{n-1})\\
             & \leq  \lambda G(x_{n-1},T^{m_n}x_{n-1}, T^{m_n}x_{n-1}) =  \lambda G(T^{m_{n-2}}x_{n-2},T^{m_n}T^{m_{n-2}}x_{n-2},T^{m_n}T^{m_{n-2}}x_{n-2}) \\
           & \leq \lambda^2 G(x_{n-2},T^{m_n}x_{n-2},T^{m_n}x_{n-2})\leq  \cdots \\
             & \leq \lambda^n G(x_0,T^{m_n}x_0,T^{m_n}x_0 ).
\end{align*}

From Lemma \ref{lem1}, it follows that 
\[
G(x_n, x_{n+1},x_{n+1}) \leq \lambda^n r(x_0).
\]
Hence for $m>n,$ we have
\[
G(x_n, x_m,x_m) = \sum_{i=n}^{m-1} G(x_i, x_{i+1},x_{i+1})\leq \frac{\lambda^n}{1-\lambda}r(x_0) \to 0 \text{ as } n \to \infty.
\]

Thus $\{x_n\}$ is a $G$-Cauchy sequence.
Moreover, since $X$ is $G$-complete there exists $\xi \in X$ such that $\{x_n\}$ $G$-converges to $\xi$.

%\newpage

\vspace*{0.3cm}

\underline{Claim:} $T\xi=\xi$

By way of contraction, assume that $T\xi\neq \xi$. Then there exists two disjoint neighborhoods $U$ and $V$ of $\xi$ and $T\xi$ respectively such that 
\begin{align*}\label{neig}
\rho = \inf \{G(x,y,y):x\in U, y\in V \}>0.
\end{align*}

Since $T$ is continuous, $x_n\in U$ and $Tx_n \in V$ for all $n$ large enough.

However, 
\begin{align*}
G(x_n, Tx_n,Tx_n) & = G(T^{m_{n-1}}x_{n-1},T^{m_{n-1}}Tx_{n-1},T^{m_{n-1}}Tx_{n-1})\\
                 & \leq  \lambda G(x_{n-1},Tx_{n-1},Tx_{n-1} ) \leq \cdots \\
                  & < \lambda^n G(x_0, Tx_0,Tx_0) \to 0 \text{ as } n \to \infty,
\end{align*}
--a contradiction since $\rho >0$. Thus $\boxed{T\xi=\xi}$.

The uniqueness of the fixed point is given for free by the inequality \eqref{BCPexteq}.

\end{proof}

The following corollary is a direct consequence of the Theorem \ref{BCPext} and is quite surprising, as result, even though very interesting.

%\newpage

\begin{corollary}\label{corBCP}

Let $T$ be a map satisfying the conditions of Theorem \ref{BCPext}, then for any initial point $x_0\in X$, the sequence of iterates $T^nx_0, n=1,2,\cdots$, $G$-converges to the unique fixed point of $T$.
\end{corollary}

\begin{proof}
According to the proof of Theorem \ref{BCPext}, there exists a unique $\xi$ such that $T\xi=\xi$. Now to show that $T^nx_0$ $G$-converges to $\xi$, we set
\[
\eta = \max \{ G(\xi,T^mx_0,T^mx_0):m=1,2,\cdots,n(\xi)-1\}.
\] 

For $n$ sufficiently large, then we know that there exists $(r,q) \in \mathbb{N}^2$ such that 

$$ 
n = r.n(\xi) +q, \ 0\leq q < n(\xi), \ r>0,
$$  

and 

\begin{align*}
G(\xi,T^nx_0,T^nx_0) & = G(T^{n(\xi)}\xi,T^{r.n(\xi) +q}x_0,T^{r.n(\xi) +q}x_0) \\
& \leq \lambda G(\xi,T^{(r-1).n(\xi) +q}x_0,T^{(r-1).n(\xi) +q}x_0)\leq \cdots \\
& \leq  \lambda^r G(\xi,T^{q}x_0,T^{q}x_0)\leq \lambda^r \eta.
\end{align*}

Moreover, since  $$\frac{n-q}{n(\xi)} = r \Longrightarrow \lim_{n\to \infty }\frac{n-q}{n(\xi)} =\lim_{n\to \infty } r = \infty,$$
we have 

\[
G(\xi,T^nx_0,T^nx_0) \leq \lambda^r \eta \to 0 \text{ as } n \to \infty,
\]
i.e. $T^nx_0$ $G$-converges to the unique fixed point to $\xi$.

\end{proof}

%As we are going towards the end of the article, 
Next, we provide an example to illustrate Theorem \ref{BCPext}. The function we consider satifies \eqref{BCPexteq} but is not a contraction\footnote{In fact, none of its powers is a contraction.} and we make use of a well-known set $X$.

\begin{example}
Let $X=[0,1]$ that we write in the form 
\[
X= \bigcup_{n=1}^{\infty} \left[ \frac{1}{2^n},\frac{1}{2^{n-1}}\right]\cup \{0\},
\]

and let's endow $X$ with the $G$-metric $d$, defined as
\[
d(x,y,z)= \max\{|x-y|,|y-z|,|z-x|\} \ \ \text{for all } x,y,z \in X.
\]
Let $T:X\to X$ be defined as follows:

\[
Tx= \begin{cases}
\frac{1}{2^{n+1}}, \qquad \qquad \qquad \quad \quad \text{   if } x\in \left[\frac{1}{2^n},\frac{3n+5}{2^{n+1}(n+2)}\right]\\
\frac{n+2}{n+3}\left( x-\frac{1}{2^{n-1}}\right)+\frac{1}{2^n}, \quad \text{ if } x\in \left[ \frac{3n+5}{2^{n+1}(n+2)},\frac{1}{2^{n-1}}\right]
\end{cases}
\]
and $T(0)=0$.

Actually $T$ maps the interval $I_n:=\left[ \frac{1}{2^n},\frac{1}{2^{n-1}}\right] $ onto the interval $I_{n+1}$.

The function $T$ is a continuous function on $[0,1]$ which leaves only $0$ fixed but is not a contraction. Moreover, straightforward computations, considering all the possible cases for $x\in I_n$ and $y\in I_m$ (with $m\geq n$ and $m\leq n$), lead to 

$$|Tx-Ty| \leq \frac{n+3}{n+4}\ |x-y| \ \ \text{ for all } y\in X.$$

Therefore, if we choose $\lambda=\frac{1}{2}$ in \eqref{BCPexteq}, then for each $x\in \left[ \frac{1}{2^n},\frac{1}{2^{n-1}}\right] $, one can take $n(x)$ to be 
$
n(x) =n+3
$ and for $x=0$, one just requires that $n(0)$ be such that $n(0)\geq 1.$

\end{example}

\section{Generalizations}

In this section, we present results which extend Theorem \ref{BCPext} along with Corollary \ref{BCPcor}. In fact, we look at mappings which are not necessarily continous, and satisfy a weaker form of \eqref{BCPexteq} for a proper subset of $X$. Moreover, we show that Theorem \ref{BCPext} remains true when the hypothesis of continuity is removed. We provide examples to illustraste the actual extensions.
The proofs we present are merely copies of the ones already done for Theorem \ref{BCPext} and Corollary \ref{BCPcor}.

\begin{lemma}\label{lemExt1}
Let $(X,G)$ be a $G$-metric space and $T:X\to X$ a mapping. Let $B\subset X$ with $T(B)\subset B$. If there exists $u\in B$ and a positive integer $n(u)$ such that $T^{n(u)}=u$ and 
\begin{equation}\label{lemext1eq}
 G(T^{n(u)}u,T^{n(u)}x,T^{n(u)}y) \leq \lambda \ G(u,x,y),
\end{equation}
for some some $\lambda<1$ and all $x,y\in B$, then $u$ is the unique fixed point of $T$ in $B$ and the sequence of iterates $T^nx_0, n=1,2,\cdots$, $G$-converges to $u$ for any initial datum $x_0\in B$.
\end{lemma}

\begin{proof}
By \eqref{lemext1eq}, it is clear that $u$ is the unique fixed point of $T^{n(u)}$ in $B$. On the other hand, observe that
\[
T(u) = T(T^{n(u)}u) = T^{n(u)}(Tu) \Longrightarrow Tu=u,
\]
and $u$ is then the fixed point of $T$ in $B$.
\end{proof}
For any initial datum $x_0\in B$, since $T(B)\subset B$, we have that $T^nx_0 \in B,$ whenever $ n=1,2,\cdots$. Let's set 
\[
\eta(x_0) = \max\{ G(u,T^mx_0,T^mx_0),m=1,2,\cdots,n(u)-1\},
\]
and for $n$ large enough, there exists $(r,s) \in \mathbb{N}^2$ such that 

$$ 
n = r.n(u) +s, \ 0\leq s < n(u), \ r>0.
$$  

Then

\begin{align*}
G(u,T^nx_0,T^nx_0) & = G(T^{n(u)}u,T^{r.n(u) +s}x_0,T^{r.n(u) +s}x_0) \\
& \leq \lambda G(u,T^{(r-1).n(u) +s}x_0,T^{(r-1).n(u) +s}x_0)\leq \cdots \\
& \leq  \lambda^r G(u,T^{s}x_0,T^{s}x_0)\leq \lambda^r \eta(x_0) \to 0 \text{ as } n \to \infty,
\end{align*}

i.e. $T^nx_0\to u \text{ as } n \to \infty$.

%\newpage

\begin{theorem}\label{thmExt}
Let $(X,G)$ be a complete $G$-metric space and $T:X\to X$ a mapping. Suppose there exists $B\subset X$ which satisfies:

\begin{enumerate}
\item[(i)] $T(B)\subset B$

\item[(ii)]
For some $0<\lambda <1$ and each $x \in B$, there exists a positive integer $n(x)\geq 1$ with 

\[
G(T^{n(x)}x,T^{n(x)}y,T^{n(x)}z) \leq \lambda \ G(x,y,z),
\]
for all $y,z\in B$,

\item[(iii)]
For some $x_0 \in B$, $cl(\{T^nx_0,n\geq 1\})\footnote{$cl$ denotes the closure with respect to the topology genererated by $G$, see \cite{Mustafa}.}\subset B.$
\end{enumerate}

Then there exists a unique $\xi \in B$such that $T\xi=\xi$ and  $T^nx_0\to \xi \text{ as } n \to \infty$ for each $x_0\in B$. 

\end{theorem}

\begin{proof}
We know from Lemma \ref{lem1} that, for any $x\in B,$
$
r(x) = \sup_n G(x, T^nx,T^nx)<\infty.
$
Let $x_0$ as in hypothesis (iii) and using (i) and (ii), let's construct the sequence of iterates 
\[
x_1 = T^{n(x_0)}(x_0), \text{ and } x_{i+1} = T^{n(x_i)}(x_i).
\]
as in the proof of Theorem \ref{BCPext}.
It is therefore easy to see, by routine calculation, that 

\[
G(x_i,x_{i+1},x_{i+1}) \leq \lambda^i G(x_0,T^{m_i}x_0,T^{m_i}x_0)\leq \lambda^i r(x_0), \qquad i\geq 1,
\]
and 

\[
G(x_i,x_{i+k},x_{i+k}) \leq \sum_{l=1}^{i+k-1}G(x_l,x_{l+1},x_{l+1}) \leq \frac{\lambda^i}{1-\lambda}r(x_0),\qquad k\geq 1.
\]
It follows that $\{x_n\}$ is a $G$-Cauchy sequence.
Moreover, since $X$ is $G$-complete, using hypothesis (iii) there exists $\xi \in B$ such that $\{x_n\}$ $G$-converges to $\xi$. Hence, there exists $n(\xi)\geq 1$ such that 

\[
G(T^{n(\xi)}\xi,T^{n(\xi)}y,T^{n(\xi)}z) \leq \lambda \ G(\xi,y,z),
\]
for all $y,z\in B$ and $\{T^{n(\xi)}x_n\}$ $G$-converges to $T^{n(\xi)}\xi$, i.e.
\[
\lim_{n\to \infty} G(x_n,T^{n(\xi)}x_n,T^{n(\xi)}x_n)= G(\xi,T^{n(\xi)}\xi,T^{n(\xi)}\xi).
\]

On the other hand 
\begin{align*}
G(x_n,T^{n(\xi)}x_n,T^{n(\xi)}x_n) & = G(T^{m_{n-1}}x_{n-1},T^{m_{n-1}}T^{n(\xi)}x_{n-1},T^{m_{n-1}}T^{n(\xi)}x_{n-1})\\
& \leq \lambda G(x_{n-1},T^{n(\xi)}x_{n-1},T^{n(\xi)}x_{n-1})\leq \cdots \\
& \leq \lambda^n G(x_{0},T^{n(\xi)}x_{0},T^{n(\xi)}x_{0})\to 0 \text{ as } n\to \infty.
\end{align*}

By Lemma \ref{lemExt1}, $\xi$ is the unique fixed point of $T$ in $B$ and $T^nx_0 \to \xi$ for any initial datum $x_0\in B.$

\end{proof}

\begin{corollary}\label{corBCPext}

Let $T$ be a map satifying the conditions of Theorem \ref{thmExt}. If moreover $G(T^{n(\xi)}\xi,T^{n(\xi)}x,T^{n(\xi)}y) \leq \lambda \ G(\xi,x,y),$ then $\xi$ is the unique fixed point of $T$ in $X$ and the sequence of iterates $T^nx_0, n=1,2,\cdots$, $G$-converges to $\xi$ for any initial datum $x_0\in X$.

\end{corollary}

\begin{proof}
The reult follows directly from Lemma \ref{lemExt1}.
\end{proof}

%We finish this manuscript by providing these two examples.
%examples to show how
%Theorem \ref{thmExt} along with Corollary \ref{corBCPext} actually extend Theorem \ref{BCPext} along with Corollary \ref{corBCP}.
We now present two examples.

\begin{example}
Let $X=[0,1]$ and let's endow $X$ with the $G$-metric $d$, defined as
\[
d(x,y,z)= \max\{|x-y|,|y-z|,|z-x|\} \ \ \text{for all } x,y,z \in X.
\]
Let $T:X\to X$ be defined as follows:

\[
Tx= \begin{cases}
1-x, \quad  \text{   if } x\in \mathbb{Q}\cap X,\\
\frac{1}{2}, \quad \qquad \text{otherwise. }
\end{cases}
\]
Then $T\left(\frac{1}{2}\right)=\frac{1}{2} $ and for any $\lambda<1$, let $B=\left\lbrace\frac{1}{2}\right\rbrace\cup A $ where $A$ is any collection of irrationals in $[0,1].$ For a rational number $y\neq \frac{1}{2}$, we have $1-y\neq \frac{1}{2}$. Moreover, $B$ cannot any rational  $x\neq \frac{1}{2}.$ 
\end{example}

\begin{example}
Let $X=[0,1]$ and let's endow $X$ with the $G$-metric $d$, defined as
\[
d(x,y,z)= \max\{|x-y|,|y-z|,|z-x|\} \ \ \text{for all } x,y,z \in X.
\]
Let $T:X\to X$ be defined as follows:

\[
Tx= \begin{cases}
x, \ \ \ \ \ \ \ \qquad  \text{   if } x\in \mathbb{Q}\cap X,\\
1-x, \quad \qquad \text{otherwise. }
\end{cases}
\]
Then each rational is a fixed point and for any $\lambda<1$, let $B=\{x\}$, $x$ rational.
\end{example}

In the last section of this paper, we shall be concerned with a triplet\footnote{The author plans to study more thoroughly and with examples common fixed point results for families of self-mappings in another paper \cite{Gaba6}.} of mappings which satisfy a contractive condition similar to the one we discussed above, namely:

let $T_1,T_2,T_3$ be self-mappings of a complete $G$-metric space $(X,G)$ such that there exists a constant $\lambda, \ 0<\lambda<1$ such that there exist positive integers $n(x),m(y),k(z)$ such that for each $x,y,z \in X,$

\begin{align}\label{eqfin}
G(T_1^{n(x)}x,T_2^{m(y)}y,T_3^{k(z)}z)  \leq &\  \lambda \max\{ A,
 B,C,D,E\},
\end{align}
where 
\begin{align*}
A &= G(x,y,z), \\
B &= G(x,T_1^{n(x)}x,T_1^{n(x)}x), \ C = G(y,T_2^{m(y)}y,T_2^{m(y)}y),\ D= G(z,T_3^{k(z)}z,T_3^{k(z)}z), \\
E &= \frac{1}{4}[G(T_1^{n(x)}x, y,z)+G(x,T_2^{m(y)}y,z) + G(x,y,T_3^{k(z)}z)].
\end{align*}

\begin{theorem}\label{thmfin}
Let $T_1,T_2,T_3$ be self-mappings on a complete $G$-metric space $(X,G)$ which satisfy \eqref{eqfin}. Then $T_1,T_2$ and $T_3$ have a unique common fixed point.
\end{theorem}

\begin{proof}
Let $x_0 \in X$, and define the sequence $\{x_n\}$ by 
\[
x_1 = T_1^{n(x_0)}x_0, x_2 = T_2^{m(x_1)}x_1, x_3 = T_3^{k(x_2)}x_2 = \cdots,\]

\[ x_{3n+1}= T_1^{n(x_{3n})}x_{3n}, \ \ x_{3n+2} = T_2^{m(x_{3n+1})}x_{3n+1}, \ \ x_{3n+3}= T_3^{k(x_{3n+2})}x_{3n+2}.
\]

Using \eqref{eqfin} and assuming, without loss of generality, that $x_n\neq x_m$ for each $n\neq m$,

\begin{align}\label{weget}
G(x_{3n+1},x_{3n+2},x_{3n+3}) \leq  \lambda \max \left\lbrace 
A'_n ,B,C,D,\frac{1}{4}[B_n+A_n+C_n]  \right\rbrace
\end{align}

with

\[
B= G(x_{3n},x_{3n+1}, x_{3n+1}), \ C = G(x_{3n+1},x_{3n+2}, x_{3n+2}), \ D= G(x_{3n+2},x_{3n+3}, x_{3n+3}),
\]

\vspace*{0.5cm}

\begin{align*}
A_n &= G(x_{3n},x_{3n+2}, x_{3n+2}), \ \ A'_n=G(x_{3n},x_{3n+1}, x_{3n+2}),\\ \\
B_n &= G(x_{3n+1},x_{3n+1}, x_{3n+2}), \ \ C_n = G(x_{3n},x_{3n+1}, x_{3n+3}). \\
\end{align*}
%G(x_{3n},x_{3n+1}, x_{3n+2}) + G(x_{3n+3},x_{3n+1}, x_{3n+2}).

First observe that, if we set $ B'_n:=G(x_{3n+3},x_{3n+1}, x_{3n+2})$, we have 

\begin{align*}
B & = G(x_{3n},x_{3n+1}, x_{3n+1}) \leq G(x_{3n},x_{3n+1},
 x_{3n+2})=A'_n, \\
 C& = G(x_{3n+1},x_{3n+2}, x_{3n+2})\leq G(x_{3n},x_{3n+1}, x_{3n+2})=A'_n
\\
D & = G(x_{3n+2},x_{3n+3}, x_{3n+3}) \leq G(x_{3n+1},x_{3n+2},x_{3n+3})=B'_n,\\ 
\end{align*}
and 

\[  
A_n \leq A'_n,  \ \ \  B_n \leq B'_n,  \ \ \text{ and   } \  C_n  \leq A'_n+ B'_n.
\]

\vspace*{0.5cm}

Therefore
\[
\frac{1}{4}[A_n+B_n+C_n] \leq \frac{1}{2}[G(x_{3n},x_{3n+1}, x_{3n+2})+G(x_{3n+3},x_{3n+1}, x_{3n+2})].
\]

\vspace*{0.5cm}

Hence, inequality \eqref{weget} becomes

\begin{align}\label{weget1}
B_n'=G(x_{3n+1},x_{3n+2},x_{3n+3}) \leq  \lambda \max \left\lbrace 
A'_n ,B_n',\frac{1}{4}[B_n+A_n+C_n]  \right\rbrace
\end{align}

which yields 
 \[
 G(x_{3n+1},x_{3n+2},x_{3n+3}) \leq  \lambda G(x_{3n},x_{3n+1}, x_{3n+2}).
 \]

%\vspace*{0.5cm}

 Indeed, if
 \[
 \max \left\lbrace A'_n,
 B_n',
\frac{1}{4}[A_n+B_n+C_n]  \right\rbrace = G(x_{3n+1},x_{3n+2},x_{3n+3})=B_n'
 \]

we get a contradiction, since $\lambda<1$. So $B_n'< A'_n.$

\vspace*{1cm}

Similarly, if

 \[
 \max \left\lbrace A'_n,
 B_n',
\frac{1}{4}[A_n+B_n+C_n]  \right\rbrace = \frac{1}{4}[A_n+B_n+C_n] \ \ \left( \leq \frac{1}{2}[ B_n'+A'_n] \right),
 \]
 
we get a contradiction because 

\[
  B_n'\leq \frac{1}{2}[ B_n'+A'_n] \leq A'_n\footnote{Remember that for two positive reals $0<\alpha< \beta$, we have
$\alpha < \frac{\alpha+\beta}{2} < \beta$.}.
\]

We then conclude that

\[
 \boxed{G(x_{3n+1},x_{3n+2},x_{3n+3}) \leq  \lambda\  G(x_{3n},x_{3n+1}, x_{3n+2})}.
 \]

In the same manner, it can be shown that 

\[
 \boxed{G(x_{3n+2},x_{3n+3},x_{3n+4}) \leq  \lambda \ G(x_{3n+1},x_{3n+2}, x_{3n+3})} ,
 \]

and 
\[
\boxed{G(x_{3n},x_{3n+1},x_{3n+2}) \leq  \lambda \ G(x_{3n-1},x_{3n},x_{3n+1})},
\]
so that 
\begin{align*}
G(x_{3n},x_{3n+1},x_{3n+2}) \leq & \ \lambda^{3n} G(x_0,x_1,x_2), \\
 G(x_{3n+1},x_{3n+2},x_{3n+3}) \leq & \  \lambda^{3n}G(x_1,x_2,x_3), \\
 G(x_{3n+2},x_{3n+3},x_{3n+4}) \leq & \ \lambda^{3n}G(x_2,x_3,x_4).
\end{align*}

Therefore, for all $n$
\[
\boxed{G(x_{n},x_{n+1},x_{n+2}) \ \leq \ \lambda^{n} \ r(x_0)},
\]

with $$r(x_0)=\max\{G(x_0,x_1,x_2),G(x_1,x_2,x_3),G(x_2,x_3,x_4) \}.$$

Hence for any $l>m>n,$ we have

\begin{align*}
G(x_n,x_m,x_l) & \leq G(x_n,x_{n+1},x_{n+1})+ G(x_{n+1},x_{n+2},x_{n+2})\\
               & +\cdots+  G(x_{l-1},x_{l-1},x_l)\\
               & \leq G(x_n,x_{n+1},x_{n+2})+ G(x_{n+1},x_{n+2},x_{n+3})\\
             & +  \cdots+  G(x_{l-2},x_{l-1},x_l)\\
             & \leq  \frac{\lambda^{n}}{1-\lambda}r(x_0).
\end{align*}

Similarly, for the cases  $l=m>n,$  and  $l>m=n,$ we have 

\[
G(x_n,x_m,x_l)  \leq \frac{\lambda^{n-1}}{1-\lambda}r(x_0).
\]

Thus is $\{x_n\}$ is $G$-Cauchy and hence $G$-converges. Call the limit $\xi.$ From \eqref{eqfin}, we have 

\begin{equation}\label{limici}
G(T_1^{n(\xi)}\xi,x_{3n+2},x_{3n+3}) = G(T_1^{n(\xi)}\xi,T_2^{m(x_{3n+1})}x_{3n+1},T_3^{k(x_{3n+2})}x_{3n+2})\leq \lambda \max\{A,B,C,D,E\}, 
\end{equation}

with

\[
A= G(\xi,x_{3n+1},x_{3n+2} ), \ \ \  B = G(\xi,T_1^{n(\xi)}\xi,T_1^{n(\xi)}\xi)\]

\[ C = G(x_{3n+1},T_2^{m(x_{3n+1})}x_{3n+1},T_2^{m(x_{3n+1})}x_{3n+1}), \ \ \ D= G(x_{3n+2},T_3^{k(x_{3n+2})}x_{3n+2},T_3^{k(x_{3n+2})}x_{3n+2}), \]

\[ E= \frac{1}{4}[G(T_1^{n(\xi)}\xi,x_{3n+1},x_{3n+2})+G(\xi,T_2^{m(x_{3n+1})}x_{3n+1},x_{3n+2} )+G(\xi,x_{3n+1},T_3^{k(x_{3n+2})}x_{3n+2})].
\]

\vspace*{0.5cm}

Taking the limit in \eqref{limici} as $n\to \infty$, we obtain

\begin{align*}
G(T_1^{n(\xi)}\xi,\xi,\xi) & \leq \lambda \max\left\lbrace 0,G(\xi,T_1^{n(\xi)}\xi,T_1^{n(\xi)}\xi),0,0,\frac{1}{4}G(T_1^{n(\xi)}\xi,\xi,\xi)\right\rbrace \\
& \leq \lambda G(T_1^{n(\xi)}\xi,T_1^{n(\xi)}\xi,\xi).
\end{align*}
If we assume, by way of contradiction that $\xi \neq T_1^{n(\xi)}\xi$, we get that 

\[
G(T_1^{n(\xi)}\xi,\xi,\xi) \leq \lambda \ G(T_1^{n(\xi)}\xi,T_1^{n(\xi)}\xi,\xi) \leq \lambda \ G(T_1^{n(\xi)}\xi,\xi,\xi).
\]
--a contradiction, hence $$\xi = T_1^{n(\xi)}\xi.$$
%, since $G$ is symmetric. 

Similarly, one shows that 
\[
T_2^{m(\xi)}\xi = \xi = T_3^{k(\xi)}\xi.
\]

Moreover, if $\eta$ is a point such that 

\[
\eta= T_1^{n(\eta)}\eta =T_2^{m(\eta)}\eta =T_3^{k(\eta)}\eta,
\]
then from \eqref{eqfin}, we can write 
\[
G(\xi,\eta,\eta) = G(T_1^{n(\xi)}\xi),T_2^{m(\eta)}\eta,T_3^{k(\eta)}\eta) \leq \lambda \max \left\lbrace G(\xi,\eta,\eta), \frac{3}{4}G(\xi,\eta,\eta)  \right\rbrace,
\]
which implies $\eta=\xi,$ i.e. $\xi$ is the unique point satisfying 
\[
a= T_1^{n(a)}a =T_2^{m(a)}a =T_3^{k(a)}a.
\]
Furthermore the condition $\xi = T_1^{n(\xi)}\xi$ implies $$T_1(\xi) = T_1(T_1^{n(\xi)}\xi)=T_1^{n(\xi)}(T_1\xi)  .$$ From the uniqueness of $\xi$, we derive that $T_1\xi=\xi$. Similarly, $$T_2\xi=\xi=T_3\xi.$$

This completes the proof.
\end{proof}

We then have the following two corollaries, for which the proofs are quite straightforward.

\begin{corollary}\label{corfin1}
Let $T$ be a self-mapping on a complete $G$-metric space $(X,G)$ such that there exists a positive real number $\lambda, 0<\lambda<1$ such that, for each $x,y,z\in X$ there exist positive integers $n(x),n(y), n(z)$ such that 

\begin{align}\label{eqfincor1}
G(T^{n(x)}x,T^{n(y)}y,T^{n(z)}z)  \leq &\  \lambda \max\{ A,
 B,C,D,E\},
\end{align}
where 
\begin{align*}
A &= G(x,y,z), \\
B &= G(x,T^{n(x)}x,T^{n(x)}x), \ C = G(y,T^{n(y)}y,T^{n(y)}y),\ D= G(z,T^{n(z)}z,T^{n(z)}z), \\
E &= \frac{1}{4}[G(T^{n(x)}x, y,z)+G(x,T^{n(y)}y,z) + G(x,y,T^{n(z)}z)].
\end{align*}

Then $T$ has a unique fixed point.
\end{corollary}

\begin{proof}
In Theorem \ref{thmfin}, set $T_1=T_2=T_3,\ m(y)=n(y)$ and $k(z)=n(z).$

\end{proof}

\begin{corollary}\label{corfin2}
Let $\{T_n\}$ be a sequence of self-mappings on a complete $G$-metric space $(X,G)$ such that there exists a positive real number $\lambda, 0<\lambda<1$ such that, for each $x,y,z\in X$ there exists positive integers $n(x),n(y), n(z)$ such that, for each $i,j,k=1,2,\cdots,$ 

\begin{align}\label{eqfincor2}
G(T_i^{n(x)}x,T_j^{n(y)}y,T_k^{n(z)}z)  \leq &\  \lambda \max\{ A,
 B,C,D,E\},
\end{align}
where 
\begin{align*}
A &= G(x,y,z), \\
B &= G(x,T_i^{n(x)}x,T_i^{n(x)}x), \ C = G(y,T_j^{n(y)}y,T_j^{n(y)}y),\ D= G(z,T_k^{n(z)}z,T_k^{n(z)}z), \\
E &= \frac{1}{4}[G(T_i^{n(x)}x, y,z)+G(x,T_j^{n(y)}y,z) + G(x,y,T_k^{n(z)}z)].
\end{align*}

Then there exists a unique common fixed point for the family $\{T_n\}$.
\end{corollary}

We here state the last theorem of this paper.

\begin{theorem}\label{lafin}

Let $\{T_n\}$ be a sequence of continuous  self-mappings of a complete $G$-metric space $(X,G)$ such that there exists a positive real number $\lambda, 0<\lambda<1$ such that, for each $x,y,z\in X$ there exists a positive integer $n(x)$ such that

 \begin{align}\label{lafineq}
G(T_i^{n(x)}x,T_i^{n(x)}y,T_i^{n(x)}z)  \leq &\  \lambda \max\{ A,
 B,C,D,E\},
\end{align}
where 
\begin{align*}
A &= G(x,y,z), \\
B &= G(x,T_i^{n(x)}x,T_i^{n(x)}x), \ C = G(y,T_i^{n(x)}y,T_i^{n(x)}y),\ D= G(z,T_i^{n(x)}z,T_i^{n(x)}z), \\
E &= \frac{1}{4}[G(T_i^{n(x)}x, y,z)+G(x,T_i^{n(x)}y,z) + G(x,y,T_i^{n(x)}z)].
\end{align*}
Suppose $\{T_i\}$ converges pointwise to a continuous function $T$. Then $T$ has a unique fixed point $x^*$. Moreover if we call $x^*_i$ the unique fixed points of the $T_i$'s, then the sequence $\{x^*_i\}$ $G$-converges to $x^*$.

\end{theorem}

\begin{proof}
In \eqref{lafineq}, take the  limit as $i\to \infty$ and use the continuity of $T$, $T_i$ and $G$ to obtain the result that $T$ satisfies \eqref{lafineq}. From Corollary \ref{corfin1}, $T$ has a unique point $x^*$.

%\eqref{eqfincor1}
%Moreover, 
%\begin{align*}
%G(x^*,x^*_i,x^*_i)& = G(T^{n(x^*_i)}x^*,T_i^{n(x^*_i)}x^*_i,T_i^{n(x^*_i)}x^*_i)\\
%& \leq G(T^{n(x^*_i)}x^*,T_i^{n(x^*_i)}x^*,T_i^{n(x^*_i)}x^*)+ G(T_i^{n(x^*_i)}x^*,T_i^{n(x^*_i)}x^*_i,T_i^{n(x^*_i)}x^*_i)
%\end{align*}

From \eqref{lafineq}, we have

\begin{equation}\label{eq1}
G(T_i^{n(x^*)}x^*,x^*_i,x^*_i)=G(T_i^{n(x^*)}x^*,T_i^{n(x^*)}x^*_i,T_i^{n(x^*)}x^*_i) \leq \lambda \max\{A ,B,C\}
\end{equation}

where 

\[
A= G(x^*,x^*_i,x^*_i), \ B=G(x^*,T_i^{n(x^*)}x^*,T_i^{n(x^*)}x^*) 
\]
\[C= \frac{1}{4} G(T_i^{n(x^*)}x^*,x^*_i,x^*_i)+ \frac{1}{2}G(x^*,x^*_i,x^*_i).
\]
Taking the limit as $i\to \infty$ in \eqref{eq1}, we obtain
\[
\lim_{i\to \infty}G(x^*,x^*_i,x^*_i) \leq \lambda\ \lim_{i\to \infty}G(x^*,x^*_i,x^*_i),
\]
which is true only if $$\lim_{i\to \infty}G(x^*,x^*_i,x^*_i)=0.$$

%%%%%%%%%%%%%%%%%%%%%%%%%%%%%%%%%%%%%%%%%%%%%%%%%%%%

%\begin{equation}\label{eq1}
%G(T_i^{n(x^*_i)}x^*,x^*_i,x^*_i)=G(T_i^{n(x^*_i)}x^*,T_i^{n(x^*_i)}x^*_i,T_i^{n(x^*_i)}x^*_i) \leq \lambda \max\{A ,B,C\}
%\end{equation}
%
%where 
%
%\[
%A= G(x^*,x^*_i,x^*_i), \ B=G(x^*,T_i^{n(x^*_i)}x^*,T_i^{n(x^*_i)}x^*) 
%\]
%\[C= \frac{1}{4} G(T_i^{n(x^*_i)}x^*,x^*_i,x^*_i)+ \frac{1}{2}G(x^*,x^*_i,x^*_i).
%\]
%Taking the limit as $i\to \infty$ in \eqref{eq1}, we obtain
%\[
%\lim_{i\to \infty}G(x^*,x^*_i,x^*_i) \leq \lambda\ \lim_{i\to \infty}G(x^*,x^*_i,x^*_i),
%\]
%which is true only if $$\lim_{i\to \infty}G(x^*,x^*_i,x^*_i)=0.$$

%%%%%%%%%%%%%%%%%%%%%%%%%%%%%%%%%%%%%%

%This completes the proof.

\end{proof}

\section{Concluding remark}

Some comments about Bryant's result can be read in \cite{sh}, where the author motivated some interesting questions regarding this modified Banach principle. More precisely, he proved that for a single valued mapping $T$ in a complete metric space $(X,d)$, if $T^n$, for some $n>1$, is a contraction, then $T$ itself\footnote{Note that the map $T$ need not be a contraction under the metric $d$.} is a contraction under another related metric $d'$. The intuition behind Theorem \ref{BCPext} is that, even though the mapping $T$ is not a contraction, locally, there is a power of $T$ which is a contraction, and that is enough for $T$ to admit a unique fixed point.

Furthermore, in the results presented in the last section of this paper, one can obviously weaken the contractive condition by restricting it to a subset $B$ of $X$ such that $B$ is invariant under the mappings involved and contains the closure of all iterates of some $x_0\in B.$

\vspace*{0.2cm}

Extending the BCP requires a structure of \textit{complete metric-like space
with contractive condition}  on the map. There is vast amount of
literature dealing with extensions/generalizations of thE BCP. An attempt was made in this manuscript to present some extensions of the BCP in which
the conclusion is obtained under mild modified conditions and which play important
role in the development of $G$-metric fixed point theory.

\section*{Conflict of interests }

The author declares that there is no conflict
of interests regarding the publication of this article.

\vspace*{0.2cm}

\section*{Acknowledgments.}
%{ \bf Acknowledgments.}

 This work was carried out with financial support from the government of Canada’s International
Development Research Centre (IDRC), and within the framework of the AIMS Research for Africa
Project.

\bibliographystyle{amsplain}

\end{document}